\documentclass[a4paper]{amsart}
\usepackage[T1]{fontenc}
\usepackage[utf8]{inputenc}
\usepackage[english]{babel}
\usepackage{verbatim}
\usepackage{amsmath}
\usepackage{amssymb}
\usepackage{amsthm}
\usepackage{tikz}
\usepackage{tikz-cd}
\usetikzlibrary{matrix} 
\usepackage{faktor}
\usepackage{graphics}
\usepackage{mathtools}
\usepackage{extarrows}
\DeclareMathAlphabet\mathbfcal{OMS}{cmsy}{b}{n}

\begin{document}
	\newcommand{\V}{\ensuremath{\mathcal{V}}} 
	\newcommand{\sV}{\ensuremath{\mathbfcal{V}} } 
	\newcommand{\Pbold}{\ensuremath{\mathbfcal{P}} }
	\newcommand{\PI}{\ensuremath{\mathbf{\Pi}} }
	\newcommand{\svec}{\ensuremath{\textbf{$\mathfrak{svec}$}} }
	\newcommand{\dg}[1]{\ensuremath{\overline{#1} \:}}
	\newcommand{\Zdz}{\ensuremath{\mathbb{Z}/2\mathbb{Z}} }
	\newcommand{\smod}{\ensuremath{\mathsf{smod}}}
	\newcommand{\modd}{\ensuremath{\mathsf{mod}}}
	\newcommand{\rbold}{\textbf{r} }
	\newcommand{\Bbold}{\textbf{B}}
	\newcommand{\Rbold}{\textbf{R}}
	\newcommand{\Ebold}[1]{\boldsymbol{E_{#1}}}
	\newcommand{\Sbold}{\textbf{S} }
	\newcommand{\Abold}{\textbf{A} }
	\newcommand{\Ibold}{\textbf{I} }
	\newcommand{\Image}{\mathrm{Im}}
	\newcommand{\sdim}{\mathrm{sdim}}
	\newcommand{\Ker}{\mathrm{Ker}}
	\newcommand{\Gammabold}{\boldsymbol{\Gamma} } 
	\newcommand{\Lambdabold}{\boldsymbol{\Lambda} } 
	\newcommand{\Omegabold}{\boldsymbol{\Omega} }
	\newcommand{\kbold}{\textbf{k} }
	\newcommand{\Tbold}{\textbf{T} }
	\newcommand{\Qbold}[1]{\ensuremath{\textbf{Q}^{#1}}}
	\newcommand{\sym}{\mathfrak{s}}
	\newcommand{\kmn}{\ensuremath{\kk^{m|n}} }
	\newcommand{\smn}[1]{\ensuremath{S(m|n, {#1})}}
	\newcommand{\sn}[1]{\ensuremath{S(n, {#1})}}
	\newcommand{\kuhn}{^{\#} }
	
	\newcommand{\ciolino}{\textbf{PROVA...}}
	
	\newcommand{\ev}{\mathrm{ev}}
	\newcommand{\Ext}{\mathrm{Ext}}
	\newcommand{\Hom}{\mathrm{Hom}}
	\newcommand{\End}{\mathrm{End}}
	\newcommand{\Ccal}{\mathcal{C}}
	\newcommand{\Pcal}{\mathcal{P}}
	\newcommand{\A}{\mathcal{A}}
	\newcommand{\B}{\mathcal{B}}
	\newcommand{\D}{\mathcal{D}}
	\newcommand{\kk}{\Bbbk}
	\newcommand{\id}{\mathrm{id}}
	\newcommand{\Id}{\mathrm{Id}}
	\newcommand{\Tot}{\mathrm{Tot}}

	\numberwithin{equation}{section}
	\theoremstyle{plain}
	\newtheorem{thm}{Theorem}[section]
	\newtheorem{lemma}[thm]{Lemma}
	\newtheorem{prop}[thm]{Proposition} 
	\newtheorem{cor}[thm]{Corollary}
	\theoremstyle{definition}
	\newtheorem{defin}[thm]{Definition} 
	\theoremstyle{remark}
	\newtheorem{rmk}[thm]{Remark} 
	\newtheorem{ex}[thm]{Example}
	\newtheorem{conj}[thm]{Conjecture}
	\newtheorem{prop-defi}[thm]{Proposition-Definition}
	\newtheorem{convention}[thm]{Convention}
	\newtheorem{notation}[thm]{Notation}
	\newtheorem{pb}[thm]{Problem}
	\DeclarePairedDelimiter\floor{\lfloor}{\rfloor}

\title{Cohomology of twisted polynomial superfunctors through the twisting spectral sequence}

\author{Iacopo Giordano}

\begin{abstract}
	We want to compute generic $\Ext$-spaces of twisted polynomial functors in relation to the $\Ext$-spaces of the untwisted ones, modulo a parametrisation. Thanks to the study of a spectral sequence we get to a computation in low degrees, with remarkable consequences at the level of generic cohomology.
\end{abstract}

\maketitle

\section{Introduction}
Strict polynomial functor categories have been kept in great consideration in the last twenty years, since Friedlander and Suslin \cite{FS} first introduced them and showed that computing cohomology inside them gave access to cohomology in the category of schemes. More recently, this category $\Pcal$ saw its \textit{super} (=$\Zdz$-graded) equivalent $\Pbold$ to come up in the paper of Axtell \cite{Axtell}, giving rise to updated questions and computations. It was an interest of the author to compute spaces of the form $\Ext_\Pbold^*(F \circ A, G \circ B)$, with $A$ and $B$ generic additive superfunctors, in relation to the bare cohomology of $F$ and $G$. He managed to find a formula \cite[Thm 4.16]{Iac} in case one of the two is additive. A generalisation for all $F,G \in \Pcal$ is conjectured at the very end of the same paper, one of its particular cases ($A=B=\Ibold_0^{(r)}$) giving a graded isomorphism natural in $F,G$:
\begin{conj}\label{but}
$\Ext^*_\Pbold(F_0^{(r)}, G_0^{(r)}) \simeq \Ext_\Pcal^* (F, G_{\Ebold{r}}) \: .$
\end{conj}
Here $\Ebold{r}$ stands for the Yoneda superalgebra $\Ext^*_\Pbold(\Ibold_0^{(r)}, \Ibold_0^{(r)})$ and $G_{\Ebold{r}}$ denotes the parametrised functor $G(\Ebold{r}\otimes -)$. As explained in detail in \cite{Iac}, $G_{\Ebold{r}}$ bears a grading, and the right side of the latter isomorphism is graded by means of the total degree. The goal of this paper is to prove a version in low degrees of this isomorphism. Such result yields an important corollary at the level of \emph{generic cohomology} of $F$ and $G$, an object that was introduced in \cite{Cline} in the category of representations of an algebraic group. We will generalise it to the context of strict polynomial superfunctors, defining the \emph{classical} generic cohomology of $F$ and $G$ to be the colimit of the diagram $\{ \Ext^*_\Pcal(F^{(r)}, G^{(r)}) \}_{r \geq 1}$ and the \emph{super} generic cohomology to be the colimit of $\{ \Ext^*_\Pbold(F_0^{(r)}, G_0^{(r)}) \}_{r \geq 1}$, both considered with the twisting maps on $\Ext$. What we prove in this paper is the following:
\begin{thm}\label{oooh}
	The isomorphism of Conjecture \ref{but} holds in degrees $* < 2p^{2r-1}$. As a consequence, the super and classical generic cohomologies of $F$ and $G$ are isomorphic.
\end{thm}

Our approach to the proof is based on the study of a suitable spectral sequence, which at its second page takes the form
\[ E_2^{s,t} = \Ext_\Pcal ^s (F, (G_{\Ebold{r}})^t) \Rightarrow \Ext^{s+t}_\Pbold(F_0^{(r)}, G_0^{(r)}) \: . \]
%
Such spectral sequence has a classical counterpart whose collapsing is known \cite{TouzeUnivSS}. Nevertheless, in the classical context we have more powerful results. In fact, an analogue of the \emph{natural} isomorphism of Conjecture \ref{but} holds for the classical groups $\Ext^*_\Pcal(F^{(r)}, G^{(r)})$ is true, as we refer in Thorem \ref{classicaliso}. The idea is then to use this information and compare the two sequences by means of a restriction morphism. 
\section{Recollections}
Throughout all the paper $\kk$ is a field of odd prime characteristic $p$ and $\V$, resp. $\sV$, is the category of finite-dimensional $\kk$-vector spaces, resp. super $\kk$-vector spaces. Let $\Pbold$ be the category of strict polynomial superfunctors and $\Pbold_d$ its full subcategory of $d$-homogenous objects. It is a \emph{$\kk$-superlinear} (i.e. $\sV$-enriched) category. An important role is played there by the super divided powers $\Gammabold^d$, in the sense of the two following standard results:
\begin{thm}[Yoneda lemma]
	Let $\Gammabold^{d,V} := \Gammabold^d \Hom(V,-)$. There exists an isomorphism of graded spaces, natural with respect to $F \in \Pbold_d$ and $V \in \sV$:
	\[ \Hom_{\Pbold_d}(\Gammabold^{d,V}, F) \simeq F(V) \: . \]
\end{thm}
\begin{cor}	For all $V \in \sV$, the object $\Gammabold^{d,V}$ is projective in $\Pbold$. Moreover, the collection $\{\Gammabold^{d,V}\}_{V \in \sV}$ forms a set of projective generators in $\Pbold$. In particular, their Kuhn duals $\Sbold^d_V$ form a set of injective cogenerators.
\end{cor} 
$\Gammabold^d$ and $\Sbold^d$ play thus the same role that $\Gamma^d$ and $S^d$ play in $\Pcal_d$. Nearly any cohomological computation we are going to do on a generic functor $F \in \Pcal_d$ will likely pass for the specific case $F=\Gamma^{d,V}$ or $F=S^d_V$ ($V \in \V$).
In the classical context, we sometimes use another set of projective generators. Set 
\[ \Lambda(n,d) := \{\lambda = (\lambda_1, \dots, \lambda_n) \mid \lambda_i \geq 0, \; \; \sum_i \lambda_i = d \} \]
and, for a fixed $\lambda \in \Lambda(n,d)$, 
\[ \Gamma^\lambda := \Gamma^{\lambda_1} \otimes \dots \otimes \Gamma^{\lambda_n} \: . \]
The exponential property of $\Gamma$ gives a decomposition (not natural in $V$)
\[ \Gamma^{d,V} \simeq \bigoplus\limits_{\lambda \in \Lambda(dim(V),d)} \Gamma^\lambda \]
which implies in particular that $\{\Gamma^\lambda\}_{\lambda \in \Lambda(n,d) }$ is another set of projective generators of $\Pcal$. In the same manner we define the $S^\lambda$.

We now recall the main concept of the paper. The $r$-th \emph{Frobenius twist} of a super space $V$ is defined as $V^{(r)} := V \otimes_\varphi \kk$, where $\varphi : \kk \rightarrow \kk
$ is the $p^r$-th power map. If $f:V \rightarrow W$ is a linear morphism, $f^{(r)} := f \otimes_\varphi 1$ is linear as a morphism $V^{(r)} \rightarrow W^{(r)}$. The formula $V \mapsto V^{(r)}$ determines a superfunctor that is strict polynomial of degree $p^r$ and noted by $\Ibold^{(r)}$. It admits two important subfunctors, whose existence \emph{a priori} is not trivial at all (see \cite[Sect. 2.7]{Drupieski} or \cite[Sect. 3]{Iac}):

\begin{prop}
	There exist subfunctors of $\Ibold^{(r)}$ defined by
	\[ \Ibold_0^{(r)} : V \longmapsto V_0^{(r)} \]
	\[ \Ibold_1^{(r)} : V \longmapsto V_1^{(r)} \]
	such that moreover $\Ibold^{(r)} = \Ibold_0^{(r)} \oplus \Ibold_1^{(r)}$.
\end{prop}
The operation we would like to perform now is "twisting" a polynomial superfunctor, i.e. precomposing it by a twist. This can be made in more than one fashion. For example, if $F \in \Pbold$, one can define $F \circ \Ibold^{(r)}$ and its subfunctors $F \circ \Ibold_0^{(r)}, F \circ \Ibold_1^{(r)}$. But, thanks to special properties of the twist, it also makes sense to precompose a \emph{classical} functor $F \in \Pcal$ by $\Ibold_0^{(r)}$ or $\Ibold_1^{(r)}$ to obtain a superfunctor \cite[Cor. 3.8]{Iac}, which we denote by $F_0^{(r)}$ and $F_1^{(r)}$ respectively. This last way of twisting is the one which will concern us through all the paper, starting with the main Conjecture \ref{but} that we have announced at the beginning. 

We conclude this section by giving a useful tool for $\Ext$ computations. The following is a super analogue of \cite[Prop. 5.2]{FS}:
\begin{thm}\label{FSext}
	Let $A,B$ strict polynomial superfunctors of degree $s,t$ respectively, such that $s+t = p^r d$. Then 
	\begin{itemize}
		\item $\Ext_\Pbold^*(A\otimes B, (S^d_V)_0^{(r)}) = 0$ if $s,t$ are not divisible by $p^r$.
		\item If $s=p^r s'$ and $t=p^r t'$ there is an isomorphism 
		\[ \Ext_\Pbold^*(A\otimes B, (S^d_V)_0^{(r)}) \simeq \Ext_\Pbold^*(A, (S^{s'}_V)_0^{(r)}) \otimes \Ext_\Pbold^*(B, (S^{t'}_V)_0^{(r)}) \]
		induced by cross product and multiplication on $S^*$.
	\end{itemize} 
\end{thm} 
\begin{cor}\label{exponentialext} Let $V \in \V$ and $X= \Gamma^{d,V}$ or $S^d_V$. Cross product and multiplication on $S^*$ induce an isomorphism
	\[ \Ext^*_\Pbold(X_0^{\lambda(r)}, (S^d_V)_0^{(r)}) \simeq \bigotimes_i \Ext^*_\Pbold(X_0^{\lambda_i(r)}, (S^{\lambda_i}_V)_0^{(r)}) \; \; .\]
\end{cor}
\section{The right adjoint of Frobenius precomposition}

We will focus on the even Frobenius twist superfunctor $\Ibold_0^{(r)}$. As said in the previous section, precomposition yields for every $d \geq 0$ an exact functor
 \begin{equation}\label{prec}
  - \circ \Ibold_0^{(r)} \: : \:  \Pcal_d \longrightarrow \Pbold_{dp^r} \; .
 \end{equation}  
We want to determine a formula for its right adjoint. In order to do so, we are going to make use of the theorems which realise $\Pbold$ and $\Pcal$ as categories of supermodules. If $m,n$ are positive integers, the \textit{Schur algebra}, resp. \textit{Schur superalgebra}, finds one of its many equivalent definitions in 
\[\sn{d} := \Gamma^d \mathrm{End}(\kk^n) \: , \]
resp.
\[\smn{d} := \Gammabold^d \mathrm{End}(\kmn) \; .\]
If $F \in \Pbold_d$, then $F(\kmn)$ is a left $\smn{d}$-supermodule via 
\begin{align*}
	\smn{d} \otimes F(\kmn) \longrightarrow & F(\kmn) \\
	\varphi \otimes v \longmapsto & (F\varphi)(v)
\end{align*}
and analogously, if $F \in \Pcal_d$, then $F(\kk^n)$ is a left $\sn{d}$-module. Moreover, if $T \in \Hom_\Pbold(F,G)$, then $T_{\kmn}$ is a $\smn{d}$-equivariant map. Analogously, if $T \in \Hom_\Pcal(F,G)$, then $T_{\kk^n}$ is $\sn{d}$-equivariant. So, evaluation at $\kmn$ and $\kk^n$ defines functors
\[\Pbold_d \longrightarrow \smn{d}-\smod \]
\[\Pcal_d \longrightarrow \sn{d}-\modd \: . \]
Both are equivalences under mild hypotheses on $n,m$: 

\begin{thm}[{\cite[Thm 4.2]{Axtell}}, {\cite[Thm 3.2]{FS}}]\label{equivfuncmod}
	Let $m,n,d$ be positive integers such that $m,n \geq d$. Then evaluation at $\kmn$ yields an equivalence of categories $\Pbold_d \simeq \smn{d}-\smod$, a quasi-inverse being provided by $M \longmapsto \Gammabold^{d, \kmn} \otimes_{\smn{d}} M$. In the classical counterpart, evaluation at $\kk^n$ yields an equivalence $\Pcal_d \simeq \sn{d}-\modd$, with quasi-inverse provided by $M \longmapsto \Gamma^{d, \kk^n} \otimes_{\sn{d}} M$.
\end{thm} 

Working in a category of (super)modules is very convenient when it comes to compute adjoints. From now on, fix $n,m \geq dp^r$. The functor (\ref{prec}) corresponds via the equivalence of Theorem \ref{equivfuncmod} to the functor
\[ \sn{d}-\modd \longrightarrow \smn{dp^r}-\smod \]
\[M \longmapsto \Gamma^{d,\kk^n} (\kk^{n (r)}) \otimes_{\sn{d}} M\]
whose right adjoint is given by a well-known general formula
\[ \smn{dp^r}-\smod \longrightarrow \smn{d}-\modd \]
\[N \longmapsto \Hom_{\smn{dp^r}}(\Gamma^{d,\kk^n} (\kk^{n (r)}), N) \; .\]
Going back up again through the equivalence of Theorem \ref{equivfuncmod} we arrive to an expression of the right adjoint of (\ref{prec}):
\begin{equation}\label{adjointprec}
	\begin{array}{c}
		\Pbold_{dp^r} \longrightarrow \Pcal_d \\
		G \longmapsto \Gamma^{d,\kk^n} \otimes_{\sn{d}} \Hom_{\smn{dp^r}}(\Gamma^{d,\kk^n} (\kk^{n (r)}), G(\kk^{n|m})) \; .
	\end{array}
\end{equation}
This last formula is quite weighty, but we can relieve it. Indeed, Theorem \ref{equivfuncmod} says in particular that for any functor $H \in \Pcal_d$ with $d \leq n$ the canonical map
\begin{center}$\Gamma^{d,\kk^n} \otimes_{\sn{d}} H(\kk^n) \longrightarrow H$\end{center}
is an isomorphism of functors. Moreover, by the same theorem, the evaluation functor is fully faithful. Thanks to these two facts we can rewrite
\begin{flalign*}
\Gamma^{d,\kk^n} \otimes_{\sn{d}} \Hom_{\smn{dp^r}}(\Gamma^{d,\kk^n} (\kk^{n (r)}), G(\kk^{n|m})) \\ \simeq \Gamma^{d,\kk^n} \otimes_{\sn{d}} \Hom_{\Pbold_{dp^r}}((\Gamma^{d,\kk^n})_0^{(r)}, G) \\ \simeq \Hom_{\Pbold_{dp^r}}((\Gamma^{d, -}) _0^{(r)}, G)
\end{flalign*}
which gives the computation we wanted:

\begin{prop}\label{rightadj}
	The right adjoint $\Pbold_{dp^r} \longrightarrow \Pcal_d$ of the precomposition functor (\ref{prec}) is given by 
	\[G \longmapsto \Hom_{\Pbold_{dp^r}}((\Gamma^{d,-})_0^{(r)}, G) \: \: .\]
\end{prop} 
\begin{rmk}
	In order to define a classical polynomial functor, the latter space in the formula of Proposition \ref{rightadj} is to be seen as a classical vector space by forgetting the $\Zdz$-grading.
\end{rmk}

\section{The spectral sequence}
As anticipated in the introduction, the graded isomorphism we desire is meant to come from the study of a spectral sequence that we introduce in this section. Its construction is highly standard, as we see in the next proposition, but having an explicit formula for the right adjoint of the twisting functor is fundamental to extract information from it.
\begin{prop}\label{appox}
	Let $\A, \B$ be abelian $\kk$-linear categories with enough projectives and injectives and $c \negmedspace: \A \rightleftarrows \B :\negmedspace \rho$ a $\kk$-linear adjunction with $c$ exact. Call by $\Rbold^*\rho$ the right derived functor of $\rho$. Then, for all $F,G \in \A$, there is a cohomological spectral sequence 
	\begin{equation}\label{cristo}
	E_2^{s,t} = \Ext_\A ^s (F, \Rbold^t \rho (cG)) \Rightarrow \Ext^{s+t}_\B(cF, cG) \; \; \; .
	\end{equation}
\end{prop}
\begin{proof}
	Let $P_*$ be a projective resolution of $F$ and $J^*$ an injective coresolution of $cG$. Consider the bicomplex $\Hom_\B (cP_*, J^*)$. Since $c$ is exact, the homology of its totalization computes $\Ext^*_\B(cF,cG)$, which is then the abutment of the associated spectral sequence as asserted. Now, since by adjunction $\Hom_\B (cP_*, J^*) \simeq \Hom_\A(P_*, \rho J^*)$, the first page of the sequence is isomorphic to $\Hom_\A(P_*, \Rbold^*\rho (cG))$. By consequence, the second page identifies to $\Ext^*_\A(F, \Rbold^*\rho (cG))$ as stated.
\end{proof}
Since it is a first quadrant spectral sequence and the $\Ext$-groups are by hypothesis $\kk$-vector spaces, there is the following standard criterion:
\begin{prop}\label{appoxxx}
The spectral sequence (\ref{cristo}) collapses at the second page if and only if there is a graded isomorphism (not necessarily natural) with respect to total degree on the right side
\[ \Ext^*_\B(cF, cG) \simeq \Ext_\A ^* (F, \Rbold^* \rho (cG)) \; \; \; .\]	
\end{prop} 
\begin{rmk}
	When $B=\Pbold$, the criterion of Proposition \ref{appoxxx} applies in the same way even if $\Pbold$ is $\kk$-superlinear (and not abelian). That is because the $\Ext$ in $\Pbold$ are computed by passing to an abelian subcategory of $\Pbold$, see \cite{Iac}.
\end{rmk}
We now apply this machinery to our context. Let $d,r\ge0$ be fixed until the end. Choose $\A = \Pcal_d$, $\B = \Pbold_{d p^r}$, $c$ the Frobenius precomposition (\ref{prec}) and $F,G \in \Pcal$. Then the spectral sequence of Proposition \ref{appox} reads
\begin{equation}\label{twistingss}
 \boldsymbol{II}_{F,G,r}^{s,t} := \Ext_{\Pcal_d} ^s (F, \Rbold^t \rho (G_0^{(r)}) ) \Rightarrow \Ext^{s+t}_{\Pbold_{d p^r}}(F_0^{(r)}, G_0^{(r)}) \; \; \; .
\end{equation}
By Proposition \ref{rightadj}, $\Rbold^* \rho (G_0^{(r)})$ is the graded polynomial functor defined for all $G \in \Pbold_{d p^r}$ by
\[V \longmapsto \Ext^*_{\Pbold}((\Gamma^{d,V})_0^{(r)}, G_0^{(r)} ) \; .\]
%
%
%
%
%
Set $\Ebold{r} := \Ext^*_\Pbold(\Ibold_0^{(r)}, \Ibold_0^{(r)})$. We will prove later (Lemma \ref{isokr}) that $\Rbold^* \rho (G_0^{(r)}) \simeq G_{\Ebold{r}}$ naturally in $G$. Thus, thanks to Prop \ref{appoxxx}, we can state our conjecture:

\begin{conj}\label{conjsequence}
	For all $F,G \in \Pcal$ the spectral sequence (\ref{twistingss}) collapses at the second page. In particular, there is a graded isomorphism (a priori not natural)
	\[\Ext^*_{\Pbold_{d p^r}}(F_0^{(r)}, G_0^{(r)}) \simeq  \Ext_{\Pcal_d}^* (F, G_{\Ebold{r}}) \; . \]
\end{conj}

We proceed to prove the validity of Conjecture \ref{conjsequence} in some particular cases. First we treat the case $G=S^d_V$. 
%
\begin{prop}\label{adjointSd}
	There is an isomorphism of graded functors, natural in $V$: 	
	\[\Rbold^* \rho ((S^d_V)_0^{(r)}) \simeq (S^d_V)_{\Ebold{r}} \; \; . \]
\end{prop}
\begin{proof}
	Consider the map
	\[\eta_{V,W} : (\Ebold{r}\otimes V \otimes W)^{\otimes d} \simeq (\Ext_{\Pbold_d}^*(\Ibold_0^{(r)} \otimes W^*, \Ibold_0^{(r)} \otimes V))^{\otimes d} \rightarrow \Ext_{\Pbold_d}^*((\Gamma^{d,W})_0^{(r)}, (S^d_V)_0^{(r)}) \]
	induced by cup product, thus natural in $V$ and $W$. We want to prove that it factors through $S^d(\Ebold{r}\otimes V \otimes W)$ and that the factored map is an isomorphism. The proof of this is made in steps.
	\begin{itemize}
		\item[\textit{Step 1:}] proof for $W=\kk$. Call $\eta_V = \eta_{V, \kk}$. By naturality, we may choose a basis of $V$ and decompose $\eta_V$ into a direct sum to prove the assertion. The source decomposes as
	\[ (\Ebold{r}\otimes V)^{\otimes d} \simeq \bigoplus_{\lambda \in \Lambda(dim(V),d)} \Ebold{r}^\lambda\]
	where each $\Ebold{r}^\lambda$ denotes a copy of $\Ebold{r}^{\otimes d}$ (we keep the index $\lambda$ for clearness). The target decomposes as
	\[\begin{array}{c c c}
	\Ext_{\Pbold_d}^*(\Gamma_0^{d (r)}, (S^d_V)_0^{(r)}) & \simeq & \bigoplus\limits_{\lambda \in \Lambda(dim(V),d)} \Ext_{\Pbold_d}^*(\Gamma_0^{d(r)}, S_0^{\lambda  (r)}) \\ \\
	& \simeq & \bigoplus\limits_{\lambda \in \Lambda(dim(V),d)} \: \bigotimes_i \Ext_{\Pbold_d}^*(\Gamma_0^{\lambda_i(r)}, S_0^{\lambda_i (r)})
	\end{array}\]
	where the second isomorphism is given by Corollary \ref{exponentialext}. Call $\eta_\lambda$ the restriction of $\eta_V$ at the $\lambda$-factor. It is then clear that $\eta_V = \bigoplus_\lambda \eta_\lambda$ and that each $\eta_\lambda$ maps onto the corresponding $\lambda$-factor on the target. Now, we know from the computations of Drupieski and Kujawa \cite[Thm 5.1.2]{DrupKujawa} that $\eta_{\mathbb{K}}$ factors into the desired isomorphism. Since each $\eta_\lambda$ identifies with a tensor power of $\eta_\kk$, the same statement holds for all $\eta_\lambda$, and thus for $\eta_V$.
	\item[\textit{Step 2:}] we prove that, for all composition $\lambda$ of lenght $n$, there is a graded isomorphism $S^\lambda(\Ebold{r}\otimes V) \simeq \Ext_{\Pbold_d}^*(\Gamma_0^{\lambda (r)}, (S^d_V)_0^{(r)})$ induced by $\eta^{\otimes n}$, cross product and multiplication on $S^*$. Indeed, Step 1 provides such isomorphism for $\lambda=(d)$. Generalisation to any composition comes from Corollary \ref{exponentialext}.
	\item[\textit{Step 3:}] proof for general $W$. Since $\eta_{V,W}$ is natural in $W$ as well, one may again choose a basis of $W$ to carry out the proof. Then $\eta_{V,W}$ decomposes as the sum of the restrictions
	\[ \eta_V^\lambda : \: \:  (\Ebold{r} \otimes V)^\lambda \longrightarrow  \Ext_{\Pbold_d}^*(\Gamma_0^{\lambda (r)}, (S^d_V)_0^{(r)})\]  
	for $\lambda$ who ranges through $\Lambda(dim(W), d)$. In fact, $\eta_V^\lambda$ is the same morphism described in Step 2. We have proved that, for all $\lambda$, it factors through $S^\lambda(\Ebold{r}\otimes V)$ into an isomorphism. It follows that so does $\eta_{V,W}$.
	\end{itemize}
\end{proof}
%
%

%
Denote by $t$ the grading on the parametrised functor $(S^d_V)_{\Ebold{r}}$. Thanks to Proposition \ref{adjointSd}, the second page of (\ref{twistingss}) for $G=S^d_V$ becomes
\[ \boldsymbol{II}_{F, S^d_V,r}^{s,t} = \Ext_{\Pcal_d}^s (F, (S^d_{V \otimes \Ebold{r}})^t) \Rightarrow \Ext^{s+t}_{\Pbold_{d p^r}}(F_0^{(r)}, (S^d_V)_0^{(r)}) \: . \]
Since $S^d_{V \otimes \Ebold{r}}$ is injective in each degree, the sequence is concentrated in a row and hence collapses at the second page. This proves Conjecture \ref{conjsequence} when $G$ is an injective cogenerator. So there exists a graded isomorphism
\begin{equation}\label{appiso}
 \Ext^*_{\Pbold_{d p^r}}(F_0^{(r)}, (S^d_V)_0^{(r)}) \simeq \Hom_{\Pcal_d}(F, S^d_{\Ebold{r} \otimes V}) \; 
\end{equation}
which is actually natural in $F$ and $V$, since in a sequence with just one row there is no filtration obstructing naturality. We use this isomorphism to prove the general identification for $\Rbold^* (\rho \circ c)$.

\begin{lemma}\label{isokr}
	There is a natural graded isomorphism of functors $\Rbold^* \rho (G_0^{(r)}) \simeq G_{\Ebold{r}} \: .$
\end{lemma}

\begin{proof} 
	Dualise and use (\ref{appiso}) to get
	\begin{align*}
	\Rbold^* \rho (G_0^{(r)}) \: (V) = \Ext^*_{\Pbold_{d p^r}}((\Gamma^{d,V})_0^{(r)}, G_0^{(r)}) \simeq \Ext^*_{\Pbold_{d p^r}}((G^\#) _0^{(r)}, (S^d_V)_0^{(r)}) \\  \simeq \Hom_{\Pcal_d}(G^\#, S^d_{\Ebold{r}\otimes V}) \simeq G(\Ebold{r} \otimes V)
	\end{align*}
	natural in either variable, which gives the result. 
\end{proof}
%
This indicates us a class of functors pairs for which Conjecture \ref{conjsequence} is true:
\begin{cor}\label{appoggione}
	Let $F,G \in \Pcal$ be such that $F$ is $G_{\Ebold{r}}$-acyclic. Then there is a graded isomorphism, natural in $F,G$:
	\[\Ext^*_{\Pbold_{d p^r}}(F_0^{(r)}, G_0^{(r)}) \simeq  \Hom_{\Pcal_d} (F, G_{\Ebold{r}}) \; .\]
\end{cor}

\begin{ex} Let $\lambda$ be a partition of $d$. Consider the \emph{Schur functor} $S_{\lambda}$ and the \emph{Weyl functor} $W_\lambda$, which are the Kuhn dual of each other. It is known (see for example \cite{Chal} or \cite[Ex. 4.8]{TouzeCoh})
that $S_\lambda$ is $(W_\mu)_V$-acyclic for all vector space $V$ and for any couple of partitions $\lambda, \mu$. Then by Corollary \ref{appoggione} we can compute:
	\[ \Ext^*_{\Pbold_{d p^r}}((W_\lambda)_0^{(r)}, (S_\mu)_0^{(r)}) \simeq  \Hom_{\Pcal_d} (W_\lambda, (S_\mu)_{\Ebold{r}}) \: .
	 \]
\end{ex}

\section{Proof of the conjecture in low degrees}

Throughout the rest of the paper $F,G$ are two strict polynomial functors of degree $d$. To lighten notation, we drop the subscripts from $\Pcal$ and $\Pbold$. We are going to prove the first part of Theorem \ref{oooh}, that is, a low-degree version of Conjecture \ref{conjsequence} which in exchange gains naturality:

\begin{thm}\label{isobasdegres}
	For all $n < 2p^{2r-1}$, the spectral sequence (\ref{twistingss}) induces an isomorphism, natural in $F,G$: 
	\[ \Ext^n_\Pbold(F_0^{(r)}, G_0^{(r)}) \simeq  \bigoplus_{s+t=n} \Ext_\Pcal^s (F, (G_{\Ebold{r}})^t) \: \: . \]
\end{thm}
%
As anticipated in the introduction, the idea is to compare the spectral sequence with the classical one of same type. More precisely, this sequence is given for all $F,G \in \Pcal$ by Proposition \ref{appox} applied to the classical groundset:
\begin{equation}\label{ssclassique}
 II_{F,G,r}^{s,t} := \Ext_\Pcal ^s (F, (G_{E_r})^t ) \Rightarrow \Ext^{s+t}_\Pcal(F^{(r)}, G^{(r)}) \end{equation}
where $E_r := \Ext_\Pcal^*(I^{(r)}, I^{(r)})$ is the classical Yoneda algebra. It is known that this sequence collapses \cite[Cor. 5]{TouzeUnivSS}, but there is a stronger result that the reader can find for example in \cite[Cor. 3.7]{DerivedKan}:
\begin{thm}\label{classicaliso}
	There is a graded natural isomorphism $\Ext_\Pcal^*(F^{(r)}, G^{(r)}) \simeq \Ext_\Pcal^*(F, G_{E_r}) \: .$
\end{thm}
In order to compare $\boldsymbol{II}^{*,*}$ and $II^{*,*}$ we will make use of a restriction morphism. Set $i_0 : \V \rightarrow \sV, \: i_0(V) = V \oplus 0$ and $u : \sV \rightarrow \V$ the functor who forgets the $\Zdz$-grading.

\begin{defin}
	If $F \in \Pbold$, define $res_0 F := u \circ F \circ i_0.$ This operation defines an exact functor $\Pbold \longrightarrow \Pcal \: .$
\end{defin}

In words, restricting a superfunctor means just to evaluate it on purely even spaces and consider the result as an ungraded space. It is clear that $res_0$ maps $\Pbold_d$ onto $\Pcal_d$.

\begin{lemma}\label{reszeroprop}
	\begin{enumerate}
		\item\label{super} For all $G \in \Pbold, \; \;$ $res_0(G_0^{(r)}) = (res_0 G)^{(r)}$.
		\item\label{classic} For all $F \in \Pcal, \; \;$ $res_0(F_0^{(r)}) = F^{(r)}.$
		\item\label{tensor} For all $G,H \in \Pbold$, $res_0(G\otimes H) \simeq res_0(G) \otimes res_0(H)$.
	\end{enumerate}
\end{lemma}
\begin{proof}
	For (\ref{super}), evaluate on $V \in \V$ to have $(res_0(G_0^{(r)})) \: (V) = G_0^{(r)} (V \oplus 0) = G(V^{(r)} \oplus 0) = (res_0 G)(V^{(r)}) = (res_0G)^{(r)} (V)$. The verification for (\ref{classic}) is analogous, while for (\ref{tensor}) is immediate from the definition of tensor of two functors.
\end{proof}

Being an exact functor,  $res_0$ induces a morphism $\Ext_\Pbold(F,G) \rightarrow \Ext_\Pcal(res_0F, res_0G)$ for all $F,G \in \Pbold$. In the following proposition we make $res_0$ into a morphism between our two spectral sequences. To make no confusion, note by $\pi_r$ the particular restriction $res_0 : \Ebold{r} \rightarrow E_r$. Remark that, as a map of graded spaces, it is the projection on the first $2p^r-1$ degrees.
\begin{prop}\label{reszeross} There exists a morphism of spectral sequences $\boldsymbol{II}_{F,G,r}^{*,*} \rightarrow II_{F,G,r}^{*,*}$ that identifies
	\begin{itemize}
		\item to $G(\pi_r)_* : \Ext_\Pcal^*(F, G_{\Ebold{r}}) \rightarrow \Ext_\Pcal^*(F, G_{E_r})$ on the second pages and
		\item to $res_0$ on the abutments.
	\end{itemize}
\end{prop}
\begin{proof} We have to manipulate the explicit construction of our spectral sequences, namely the one made in general in the proof of Proposition \ref{appox} and its classical counterpart. Let $J^* \hookleftarrow G_0^{(r)}$ be an injective coresolution and $P_* \twoheadrightarrow F$ be a projective resolution. Then the spectral sequence $\boldsymbol{II}_{F,G,r}^{*,*}$ is induced by the bicomplex\footnote{For convenience, we do not pass immediately to the right adjoint $\rho$ as we did in the proof of Proposition \ref{appox}.} 
$\Hom_\Pbold((P_s)_0^{(r)}, J^t)$. We proceed to define our restriction morphism of spectral sequences.
Take $K^* \hookleftarrow G^{(r)}$ an injective coresolution in $\Pcal$. Since $res_0(J^*)$ is a (not injective) resolution of $G^{(r)}$, there exists a morphism of complexes $res_0(J^*) \rightarrow K^*$ lifting the identity. Define the restriction morphism at the page zero by the composite 
  \begin{equation}\label{restrbicomplex}
  \Hom_\Pbold((P_*)_0^{(r)}, J^*) \xrightarrow{res_0} \Hom_\Pcal(P_*^{(r)}, res_0(J^*)) \rightarrow \Hom_\Pcal(P_*^{(r)}, K^*) \; . 
  \end{equation}
%
%
%
%
%
%
%
Note that the right-most bicomplex gives rise to $II^{*,*}_{F,G,r}$. The morphism (\ref{restrbicomplex}) identifies to the restriction of extensions between the abutments $\Ext^*_\Pbold(F_0^{(r)}, G_0^{(r)})\rightarrow \Ext^*_\Pcal(F^{(r)}, G^{(r)})$ and between the first pages 
\begin{equation}\label{restrappoggio}
\Ext^t_\Pbold((P_s)_0^{(r)}, G_0^{(r)})\rightarrow \Ext^t_\Pcal((P_s)^{(r)}, G^{(r)}) \: .
\end{equation}
Now, by (\ref{appiso}) there is an isomorphism
\begin{equation}\label{eheh}
\Ext^*_\Pbold((P_s)_0^{(r)}, G_0^{(r)}) \simeq \Hom_\Pcal(P_s, G_{\Ebold{r}}) 
\end{equation}
by which we compute the second page. In the classical case, there is an analogue isomorphism $\Ext^*_\Pcal(P_s^{(r)}, G^{(r)}) \simeq \Hom_\Pcal(P_s, G_{E_{r}})$. Hence, to conclude we have to verify that (\ref{restrappoggio}) identifies via (\ref{eheh}) to the push-forward $G(\pi_r)_* : \Hom_\Pcal(P_s, (G_{\Ebold{r}})^t) \rightarrow \Hom_\Pcal(P_s, (G_{E_r})^t)$. It will suffice to do that for $G$ injective, so by Corollary \ref{exponentialext} we can suppose $P_s:=\Gamma^d$ and $G:=S^d_V$. In that case, the isomorphism $S^d(V \otimes \Ebold{r}) \simeq \Hom_\Pcal(\Gamma^d, (S^d_V)_{\Ebold{r}})\simeq \Ext^*_\Pbold(\Gamma_0^{d(r)}, (S^d_V)_0^{(r)})$, as well as its classical analogue, is given explicitly by the cup product $e_1 \cdot \dots \cdot e_d \longmapsto e_1 \cup \dots \cup e_d$ (we push elements of $V$ inside the extensions of $\Ebold{r}$). Lemma \ref{reszeroprop}(\ref{tensor}) implies that $res_0$ commutes with the cup product, hence the restriction morphism on the first pages identifies to $S^d_V(\Ebold{r}) \rightarrow S^d_V(E_r)$ induced by $\pi_r$. 
\end{proof}
The restriction morphism gives us a good tool to compare the two sequences. Since $II_{F,G,r}^{*,*}$ collapses at the second page, we have that $G(\pi_r)_* \circ \delta = 0$ for any differential $\delta$ at any page of $\boldsymbol{II}_{F,G,r}^{*,*}$. What we want to study is then the kernel of $G(\pi_r)_*$. Let us make some definitions. Consider the set of \emph{infinite} compositions of the integer $d$:
\[ \Lambda(\infty,d) = \{\lambda=(\lambda_i)_{i \in \mathbb{N}} \mid \lambda_i \geq 0, \: \sum_{i \geq 0} \lambda_i = d \}  \]
and define the weight of $\lambda$ to be $|\lambda| := \sum_{i \geq 0} 2 i \lambda_i$ (this makes sense because $\lambda$ can only have finitely many nonzero components). In this sense, there is a graded decomposition 
\[G_{\Ebold{r}} = \bigoplus_{\lambda \in \Lambda(\infty,d)} G^\lambda \: \: .\]
Now say that $\lambda$ is \emph{$n$-bounded} if it is nonzero only in components $\lambda_i$ for $i < n$. Then
\[G(\pi_r)\vert_{G^\lambda} = \begin{cases*}
\text{isomorphism on the image} & if $\lambda$ is $p^r$-bounded \\
0 & otherwise 
\end{cases*}  \]
thus the same holds for the push forward $G(\pi_r)_*$. The property of being bounded is partially controlled by the weight:
\begin{lemma}\label{padrelivio}
	If $|\lambda|<2n$, then $\lambda$ is $n$-bounded. 
\end{lemma} 
\begin{proof}
	The non-$n$-bounded composition with minimal weight is ($d-1, 0, \dots,0, 1, 0, \dots)$ with $1$ in place $n$. Its weight is $2n$, which proves the assertion.
\end{proof}

\begin{cor}\label{ciolino}
	If $t<2p^r$, the restriction of $res_0 : (G_{\Ebold{r}})^t \rightarrow (G_{E_r})^t$ is an isomorphism.
\end{cor}
\begin{proof}
	$(G_{\Ebold{r}})^t = \bigoplus_{|\lambda|=t}G^\lambda$, so the hypothesis implies that all the compositions $\lambda$ in the sum have weight $< 2p^r$. By Lemma \ref{padrelivio} they are all $p^r$-bounded, hence the restriction of $res_0$ is an isomorphism on its image. But the image is the sum of the $G^{\lambda}$ with $\lambda$ $p^r$-bounded, which coincides with $(G_{E_r})^t.$
\end{proof}

Instead of applying this directly on the spectral sequence, we make a more optimal choice on the parameters $F,G,r$. Fix definitively $F,G$ and $r \geq 1$ till the end of the section and set:
\[ \boldsymbol{II}^{s,t} := \boldsymbol{II}_{F^{(r-1)},G^{(r-1)},1}^{s,t} = \Ext_\Pcal ^s (F^{(r-1)}, ((G^{(r-1)})_{\Ebold{1}} )^t) \Rightarrow \Ext^{s+t}_\Pbold(F_0^{(r)}, G_0^{(r)})   \]
\[
II^{s,t} := II_{F^{(r-1)},G^{(r-1)},1}^{s,t} = \Ext_\Pcal ^s (F^{(r-1)}, ((G^{(r-1)})_{E_1} )^t) \Rightarrow \Ext^{s+t}_\Pcal(F^{(r)}, G^{(r)}) \: .
\]
Since $(G^{(r-1)})_{\Ebold{1}} = (G_{\Ebold{1}^{(r-1)}})^{(r-1)}$, it decomposes as the sum of $G^\lambda$, $\lambda \in \Lambda(\infty,d)$, where the weight of $G^\lambda$ is $||\lambda|| := 2p^{2(r-1)}\sum_i i \lambda_i$. 
In total analogy with Lemma \ref{padrelivio} we have the following criterion:
\begin{lemma}\label{ooooh}
	If $||\lambda|| < 2p^{2r-1}$, then $\lambda$ is $p$-bounded. 
\end{lemma}
%
We are now ready to prove the theorem that we announced at the beginning of the section.

\begin{proof}[Proof of Theorem \ref{isobasdegres}]
	By Lemma \ref{ooooh} it follows that the morphism $G^{(r-1)}(\pi_1)_* : \Ext^*_\Pcal(F^{(r-1)}, ((G^{(r-1)})_{\Ebold{1}})^t) \rightarrow \Ext^*_\Pcal(F^{(r-1)}, ((G^{(r-1)})_{E_1})^t)$ is an isomorphism for all $t < 2p^{2r-1}$. By Proposition \ref{reszeross} it is a spectral sequence morphism which identifies to $res_0 : \Ext_\Pbold^*(F_0^{(r)}, G_0^{(r)}) \rightarrow \Ext_\Pcal^*(F^{(r)}, G^{(r)})$ between the abutments, so the latter is also an isomorphism in degrees $* < 2p^{2r-1}$. Consequently, in such degrees we have a chain of natural isomorphisms
	\[ \Ext_\Pbold^*(F_0^{(r)}, G_0^{(r)}) \simeq \Ext_\Pcal^*(F^{(r)}, G^{(r)}) \simeq \Ext_\Pcal^*(F^{(r-1)}, (G^{(r-1)})_{E_1}) \simeq \Ext_\Pcal^*(F^{(r-1)}, (G^{(r-1)})_{\Ebold{1}}) \]
	where the second one comes from Theorem \ref{classicaliso} and the last one is $G(\pi_1)_*$. Since $(G^{(r-1)})_{\Ebold{1}} = (G_{\Ebold{1}^{(r-1)}})^{(r-1)}$, applying again Theorem \ref{classicaliso} one gets
	\[ \Ext_\Pcal^*(F^{(r-1)}, (G^{(r-1)})_{\Ebold{1}}) = \Ext_\Pcal^*(F, G_{\Ebold{1}^{(r-1)}\otimes E_{r-1}}) \; . \]
	It is straightforward to verify that $\Ebold{1}^{(r-1)}\otimes E_{r-1} \simeq \Ebold{r}$, which ends the proof.
\end{proof}

Define the generic cohomology in the two categories:
\[ H_\Pcal^*(F,G) := colim_{r\geq 1} \: \Ext^*_\Pcal(F^{(r)}, G^{(r)}) \]
\[ H_\Pbold^*(F,G) := colim_{r\geq 1} \: \Ext^*_\Pbold(F_0^{(r)}, G_0^{(r)}) \]
the diagrams being taken along with the respective twisting maps between $\Ext$-groups. 

\begin{cor}
	For all $F,G \in \Pcal$ there is a natural isomorphism
	\[ H_\Pbold^*(F,G) \simeq H_\Pcal^*(F,G) \: .\]
\end{cor}
\begin{proof}
	By Theorem \ref{isobasdegres} and Corollary \ref{ciolino} it follows that $res_0$ induces an isomorphism $\Ext_\Pbold^t(F_0^{(r)}, G_0^{(r)}) \simeq \Ext_\Pcal^t(F^{(r)}, G^{(r)})$ for all $t < 2p^r$. This isomorphism implies the assertion on the colimits, provided that we show that $res_0$ commutes with the twisting maps. That means, for all $r\geq 1$ the diagram
	\[
	\begin{tikzcd}
	\Ext_\Pbold^*(F_0^{(r)}, G_0^{(r)}) \arrow{r}{res_0} \arrow{d}{-_0^{(1)}} & \Ext_\Pcal^*(F^{(r)}, G^{(r)} \arrow{d}{-^{(1)}}) \\
	\Ext_\Pbold^*(F_0^{(r+1)}, G_0^{(r+1)}) \arrow{r}{res_0} & \Ext_\Pcal^*(F^{(r+1)}, G^{(r+1)})
	\end{tikzcd}
	\]
	must commute. This is an immediate consequence of Lemma \ref{reszeroprop}(\ref{super}).
\end{proof}
\bibliographystyle{plain}
\bibliography{biblio.bib}
\end{document}